\documentclass{amsart}
\usepackage{amssymb,amsmath,amsthm}
\usepackage[dvips]{graphicx}

\let\oldmarginpar\marginpar
\renewcommand\marginpar[1]{\-\oldmarginpar[\raggedleft\footnotesize #1]%
{\raggedright\footnotesize #1}}

\begin{document}

\newtheorem{theorem}{Theorem}[section]
\newtheorem{corollary}[theorem]{Corollary}
\newtheorem{lemma}[theorem]{Lemma}
\newtheorem{proposition}[theorem]{Proposition}
\theoremstyle{definition}
\newtheorem{definition}[theorem]{Definition}
\theoremstyle{remark}
\newtheorem{remark}[theorem]{Remark}
\theoremstyle{definition}
\newtheorem{example}[theorem]{Example}

\def\rank{{\text{rank}\,}}

\numberwithin{equation}{section}

\title{Semi-slant Riemannian maps}

\author{Kwang-Soon Park}
\address{Department of Mathematical Sciences, Seoul National University, Seoul 151-747, Republic of Korea}
\email{parkksn@gmail.com}



\begin{abstract}
As a generalization of slant submersions \cite{S},
semi-slant submersions \cite{PP}, and slant Riemannian maps \cite{S4}, we define the notion of semi-slant Riemannian maps from
almost Hermitian manifolds to Riemannian manifolds. We study the integrability of distributions, the geometry of fibers,
the harmonicity of such maps, etc. We also find a condition for such maps to be totally geodesic and investigate some decomposition theorems.
Moreover, we give examples.
\end{abstract}

\maketitle
\section{Introduction}\label{intro}
\addcontentsline{toc}{section}{Introduction}

Given a $C^{\infty}$-map $F$ from a Riemannian manifold  $(M,g_M)$ to a Riemannian manifold $(N,g_N)$, according to the conditions on the map $F$,
we call $F$ a harmonic map \cite{BW}, a totally geodesic map \cite{BW}, an isometric immersion \cite{C}, a Riemannian submersion \cite{O},
a Riemannian map \cite{F}, etc.

For the isometric immersion $F$, it is originated from Gauss' work, which studied surfaces in the Euclidean space $\mathbb{R}^3$ and there are a lots
of papers and books on this topic. With $F$ to be a Riemannian submersion, B. O'Neill \cite{O} and A. Gray \cite{G} firstly studied the map $F$.
And there are several kinds of Riemannian submersions:

semi-Riemannian submersion and Lorentzian submersion \cite{FIP},
Riemannian submersion (\cite{G}, \cite{O}), slant submersion
(\cite{C}, \cite{S}), anti-invariant submersion \cite{S3}, almost Hermitian submersion \cite{W},
contact-complex submersion \cite{IIMV}, quaternionic submersion
\cite{IMV}, almost h-slant submersion
\cite{P}, semi-invariant submersion \cite{S2}, almost h-semi-invariant
submersion  \cite{P2}, semi-slant submersion \cite{PP}, almost h-semi-slant submersions
 \cite{P3}, v-semi-slant submersions \cite{P4}, almost h-v-semi-slant submersions
 \cite{P5}, etc.

A. Fischer \cite{F} introduced a Riemannian map $F$, which generalizes and unifies the notions of an isometric immersion, a Riemannian submersion, and an isometry. After that, there are lots of papers on this topic. Moreover, B. Sahin defined a slant Riemannian map \cite{S4}. As a generalization of slant Riemannian maps \cite{S4} and semi-slant submersions \cite{PP}, we will define a semi-slant Riemannian map.

The paper is organized as follows. In section 2 we remind some notions, which are needed for later use. In section 3 we give the definition of a semi-slant Riemannian map and obtain some properties on it. In section 4 using a semi-slant Riemannian map, we obtain some decomposition theorems. In section 5  we give examples.

\section{Preliminaries}\label{prelim}

Let $(M,g_M)$ and $(N,g_N)$ be Riemannian manifolds, where $M$, $N$ are $C^{\infty}$-manifolds and $g_M$, $g_N$ are Riemannian metrics on $M$, $N$, respectively. Let $F : (M,g_M) \mapsto (N,g_N)$ be a $C^{\infty}$-map. We call the map $F$ a {\em $C^{\infty}$-submersion} if $F$ is surjective and the differential $(F_*)_p$  has a maximal rank for any $p\in M$. The map $F$ is said to be a {\em Riemannian submersion} \cite{O} if $F$ is a $C^{\infty}$-submersion and $(F_*)_p : ((\ker (F_*)_p)^{\perp}, (g_M)_p) \mapsto (T_{F(p)} N, (g_N)_{F(p)})$ is a linear isometry for each $p\in M$, where $(\ker (F_*)_p)^{\perp}$ is the orthogonal complement of the space $\ker (F_*)_p$ in the tangent space $T_p M$ of $M$ at $p$. We call the map $F$ a {\em Riemannian map} \cite{F} if $(F_*)_p : ((\ker (F_*)_p)^{\perp}, (g_M)_p) \mapsto ((range F_*)_{F(p)}, (g_N)_{F(p)})$ is a linear isometry for each $p\in M$, where $(range F_*)_{F(p)} := (F_*)_p ((\ker (F_*)_p)^{\perp})$ for $p\in M$. Let $(M,g_M,J)$ be an almost Hermitian manifold and $(N,g_N)$ a Riemannian manifold, where $J$ is an almost complex structure on $M$. Let $F : (M,g_M,J) \mapsto (N,g_N)$ be a $C^{\infty}$-map. We call the map $F$ a {\em slant submersion} \cite{S} if $F$ is a Riemannian submersion
and the angle $\theta = \theta (X)$ between $JX$ and the space $\ker (F_*)_p$ is constant for nonzero $X\in \ker (F_*)_p$
and $p\in M$.

We call the angle $\theta$ a {\em slant angle}.

The map $F$ is said to be an {\em anti-invariant submersion}
\cite{S3} if $F$ is a Riemannian submersion and  $JX\in \Gamma((\ker F_*)^{\perp})$ for $X\in \Gamma(\ker F_*)$.
We call the map $F$ a {\em semi-invariant submersion} \cite{S2} if $F$ is a Riemannian submersion and  there is a distribution
$\mathcal{D}_1 \subset \ker F_*$ such that
$$
\ker F_*=\mathcal{D}_1\oplus \mathcal{D}_2, \
 J(\mathcal{D}_1)=\mathcal{D}_1, \ J(\mathcal{D}_2)\subset (\ker
F_*)^{\perp},
$$
where $\mathcal{D}_2$ is the orthogonal complement of
$\mathcal{D}_1$ in $\ker F_*$.
The map $F$ is said to be a {\em semi-slant submersion}
\cite{PP} if $F$ is a Riemannian submersion and there is a
distribution $\mathcal{D}_1\subset \ker F_*$ such that
$$
\ker F_* =\mathcal{D}_1\oplus \mathcal{D}_2, \
J(\mathcal{D}_1)=\mathcal{D}_1,
$$
and the angle $\theta=\theta(X)$ between $JX$ and the space
$(\mathcal{D}_2)_q$ is constant for nonzero $X\in
(\mathcal{D}_2)_q$ and $q\in M$, where $\mathcal{D}_2$ is the
orthogonal complement of $\mathcal{D}_1$ in $\ker F_*$.

We call the angle $\theta$ a {\em semi-slant angle}.

We call the map $F$ a {\em slant Riemannian map} \cite{S4} if $F$ is a Riemannian map and the angle $\theta = \theta (X)$ between $JX$
and the space $\ker (F_*)_p$ is constant for nonzero $X\in \ker (F_*)_p$ and $p\in M$.

We call the angle $\theta$ a {\em slant angle}.

The map $F$ is said to be a {\em semi-invariant Riemannian map} \cite{S5} if $F$ is a Riemannian map and  there is a distribution
$\mathcal{D}_1 \subset \ker F_*$ such that
$$
\ker F_*=\mathcal{D}_1\oplus \mathcal{D}_2, \
 J(\mathcal{D}_1)=\mathcal{D}_1, \ J(\mathcal{D}_2)\subset (\ker
F_*)^{\perp},
$$
where $\mathcal{D}_2$ is the orthogonal complement of
$\mathcal{D}_1$ in $\ker F_*$.
Let $F : (M, g_M) \mapsto (N, g_N)$ be
a $C^{\infty}$-map. The second fundamental form of $F$ is given by
\begin{equation}\label{eq: second}
(\nabla F_*)(X,Y) := \nabla^F _X F_* Y-F_* (\nabla _XY) \quad
\text{for} \ X,Y\in \Gamma(TM),
\end{equation}
where $\nabla^F$ is the pullback connection and we denote
conveniently by $\nabla$ the Levi-Civita connections of the
metrics $g_M$ and $g_N$ \cite{BW}. Remind that $F$ is said to be {\em
harmonic} if we have the tension field $\tau (F) := trace (\nabla F_*)=0$ and we call the map $F$ a {\em
totally geodesic} map if $(\nabla F_*)(X,Y)=0$ for $X,Y\in \Gamma
(TM)$ \cite{BW}.
Denote the range of $F_*$ by $range F_*$ as a subset of the pullback bundle $F^{-1} TN$. With its orthogonal complement $(range F_*)^{\perp}$
we have the following decomposition
$$
F^{-1} TN = range F_* \oplus (range F_*)^{\perp}.
$$
Moreover, we get
$$
TM = \ker F_* \oplus (\ker F_*)^{\perp}.
$$
Then we easily have

\begin{lemma}\label{lem: hori}
Let $F$ be a Riemannian map from a Riemannian manifold $(M,g_M)$ to a Riemannian manifold $(N,g_N)$. Then
$$
(\nabla F_*)(X,Y) \in \Gamma((range F_*)^{\perp}) \quad \text{for} \ X,Y\in \Gamma((\ker F_*)^{\perp}).
$$
\end{lemma}

\begin{lemma}
Let $F$ be a Riemannian map from a Riemannian manifold $(M,g_M)$ to a Riemannian manifold $(N,g_N)$. Then
the map $F$ satisfies a generalized eikonal equation \cite{F}
$$
2e(F) = ||F_*||^2 = \rank F.
$$
\end{lemma}

As we know, $||F_*||^2$ is a continuous function on $M$ and $\rank F$ is integer-valued so that
$\rank F$ is locally constant. Hence, if $M$ is connected, then $\rank F$ is a constant function \cite{F}.

\section{Semi-slant Riemannian maps}\label{semi}

\begin{definition}
Let $(M,g_M,J)$ be an almost Hermitian manifold and $(N,g_N)$ a
Riemannian manifold. A Riemannian map $F : (M,g_M,J)\mapsto
(N,g_N)$ is called a {\em semi-slant Riemannian map} if there is a
distribution $\mathcal{D}_1\subset \ker F_*$ such that
$$
\ker F_* =\mathcal{D}_1\oplus \mathcal{D}_2, \
J(\mathcal{D}_1)=\mathcal{D}_1,
$$
and the angle $\theta=\theta(X)$ between $JX$ and the space
$(\mathcal{D}_2)_p$ is constant for nonzero $X\in
(\mathcal{D}_2)_p$ and $p\in M$, where $\mathcal{D}_2$ is the
orthogonal complement of $\mathcal{D}_1$ in $\ker F_*$.
\end{definition}

We call the angle $\theta$ a {\em semi-slant angle}.

Let $F : (M,g_M,J)\mapsto (N,g_N)$ be a semi-slant
Riemannian map. Then there is a distribution $\mathcal{D}_1\subset
\ker F_*$ such that
$$
\ker F_* =\mathcal{D}_1\oplus \mathcal{D}_2, \
J(\mathcal{D}_1)=\mathcal{D}_1,
$$
and the angle $\theta=\theta(X)$ between $JX$ and the space
$(\mathcal{D}_2)_p$ is constant for nonzero $X\in
(\mathcal{D}_2)_p$ and $p\in M$, where $\mathcal{D}_2$ is the
orthogonal complement of $\mathcal{D}_1$ in $\ker F_*$.

Then for $X\in \Gamma(\ker F_*)$, we get
\begin{equation}\label{eq: proj1}
X = PX+QX,
\end{equation}
where $PX\in \Gamma(\mathcal{D}_1)$ and $QX\in
\Gamma(\mathcal{D}_2)$.

For $X\in \Gamma(\ker F_*)$, we write
\begin{equation}\label{eq: proj2}
JX = \phi X+\omega X,
\end{equation}
where $\phi X\in \Gamma(\ker F_*)$ and $\omega X\in \Gamma((\ker
F_*)^{\perp})$.

For $Z\in \Gamma((\ker F_*)^{\perp})$, we have
\begin{equation}\label{eq: proj3}
JZ = BZ+CZ,
\end{equation}
where $BZ\in \Gamma(\ker F_*)$ and $CZ\in \Gamma((\ker
F_*)^{\perp})$.

For $U\in \Gamma(TM)$, we obtain
\begin{equation}\label{eq: proj4}
U = \mathcal{V}U+\mathcal{H}U,
\end{equation}
where $\mathcal{V}U\in \Gamma(\ker F_*)$ and $\mathcal{H}U\in
\Gamma((\ker F_*)^{\perp})$.

For $W\in \Gamma(F^{-1}TN)$, we write
\begin{equation}\label{eq: proj5}
W = \bar{P}W+\bar{Q}W,
\end{equation}
where $\bar{P}W\in \Gamma(range F_*)$ and $\bar{Q}W\in \Gamma((range F_*)^{\perp})$.

Then
\begin{equation}\label{eq: proj6}
(\ker F_*)^{\perp} = \omega \mathcal{D}_2 \oplus \mu,
\end{equation}
where $\mu$ is the orthogonal complement of $\omega \mathcal{D}_2$
in $(\ker F_*)^{\perp}$ and is invariant  under $J$.

Furthermore,
{\setlength\arraycolsep{2pt}
\begin{eqnarray*}
& & \phi \mathcal{D}_1 = \mathcal{D}_1, \omega \mathcal{D}_1 = 0, \phi \mathcal{D}_2 \subset \mathcal{D}_2,
B((\ker F_*)^{\perp}) = \mathcal{D}_2        \\
& & \phi^2+B\omega = -id, C^2+\omega B = -id, \omega \phi +C\omega
= 0, BC+\phi B = 0.
\end{eqnarray*}}

Define the tensors $\mathcal{T}$ and $\mathcal{A}$ by
\begin{eqnarray}
  \mathcal{A}_E F & = & \mathcal{H}\nabla_{\mathcal{H}E} \mathcal{V}F+\mathcal{V}\nabla_{\mathcal{H}E} \mathcal{H}F \label{eq: oten1} \\
   \mathcal{T}_E F & = & \mathcal{H}\nabla_{\mathcal{V}E} \mathcal{V}F+\mathcal{V}\nabla_{\mathcal{V}E} \mathcal{H}F \label{eq: oten2}
\end{eqnarray}
for $E, F\in \Gamma(TM)$, where $\nabla$ is the Levi-Civita
connection of $g_M$.

For $X,Y\in \Gamma(\ker F_*)$, define
\begin{equation}\label{eq: vert}
\widehat{\nabla}_X Y := \mathcal{V}\nabla_X Y
\end{equation}
\begin{eqnarray}
(\nabla_X \phi)Y & := & \widehat{\nabla}_X \phi Y-\phi
\widehat{\nabla}_X Y \label{eq: vertc}  \\
(\nabla_X \omega)Y & := & \mathcal{H}\nabla_X \omega
Y-\omega\widehat{\nabla}_X Y. \label{eq: horizc}
\end{eqnarray}
Then it is easy to obtain

\begin{lemma}
Let $(M,g_M,J)$ be a K\"{a}hler manifold and $(N,g_N)$ a
Riemannian manifold. Let $F : (M,g_M,J) \mapsto (N,g_N)$ be a
semi-slant Riemannian map. Then we get
\begin{enumerate}
\item
\begin{align*}
  &\widehat{\nabla}_X \phi Y+\mathcal{T}_X \omega Y = \phi\widehat{\nabla}_X Y+B\mathcal{T}_X Y    \\
  &\mathcal{T}_X \phi Y+\mathcal{H}\nabla_X \omega Y =
  \omega\widehat{\nabla}_X Y+C\mathcal{T}_X Y
\end{align*}
for $X,Y\in \Gamma(\ker F_*)$.
\item
\begin{align*}
  &\mathcal{V}\nabla_Z BW+\mathcal{A}_Z CW = \phi\mathcal{A}_Z W+B\mathcal{H}\nabla_Z W    \\
  &\mathcal{A}_Z BW+\mathcal{H}\nabla_Z CW = \omega\mathcal{A}_Z
  W+C\mathcal{H}\nabla_Z W
\end{align*}
for $Z,W\in \Gamma((\ker F_*)^{\perp})$.
\item
\begin{align*}
  &\widehat{\nabla}_X BZ+\mathcal{T}_X CZ = \phi\mathcal{T}_X Z+B\mathcal{H}\nabla_X Z    \\
  &\mathcal{T}_X BZ+\mathcal{H}\nabla_X CZ = \omega\mathcal{T}_X Z+C\mathcal{H}\nabla_X Z  \\
  &\mathcal{V}\nabla_Z \phi X + \mathcal{A}_Z \omega X = \phi \mathcal{V}\nabla_Z X + B\mathcal{A}_Z X     \\
  &\mathcal{A}_Z \phi X + \mathcal{H}\nabla_Z \omega X = \omega\mathcal{V}\nabla_Z X + C\mathcal{A}_Z X
\end{align*}
for $X\in \Gamma(\ker F_*)$ and $Z\in \Gamma((\ker F_*)^{\perp})$.
\end{enumerate}
\end{lemma}

Let $F$ be a slant Riemannian map from an almost Hermitian manifold $(M,g_M,J)$ to a Riemannian manifold $(N,g_N)$ with
the slant angle $\theta$ \cite{S4}. Then given non-vanishing $X\in \Gamma(\ker F_*)$, we have
$$
\cos \theta = \frac{|\phi X|}{|JX|} \quad \text{and} \quad \cos \theta = \frac{g_M(JX, \phi X)}{|JX|\cdot|\phi X|} = \frac{-g_M(X, \phi^2 X)}{|X|\cdot|\phi X|}
$$
so that
$$
\cos^2 \theta = \frac{-g_M(X, \phi^2 X)}{|X|^2},
$$
which means
\begin{equation}\label{eq: angle}
\phi^2 X = -\cos^2 \theta\cdot X.
\end{equation}
Furthermore, if $(M,g_M,J)$ is K\"{a}hler, then it is easy to get
\begin{eqnarray}
  (\nabla_X \omega)Y & = & C\mathcal{T}_X Y-\mathcal{T}_X \phi Y  \label{eq: vertc2}   \\
  (\nabla_X \phi)Y & = & B\mathcal{T}_X Y-\mathcal{T}_X \omega Y  \label{eq: horizc2}
\end{eqnarray}
for $X,Y\in \Gamma(\ker F_*)$.
Assume that the tensor $\omega$ is parallel.

Then
$$
C\mathcal{T}_X Y = \mathcal{T}_X \phi Y \quad \text{for} \ X,Y\in \Gamma(\ker F_*)
$$
so that interchanging the role of $X$ and $Y$,
$$
C\mathcal{T}_Y X = \mathcal{T}_Y \phi X \quad \text{for} \ X,Y\in \Gamma(\ker F_*).
$$
Hence,
$$
\mathcal{T}_X \phi Y = \mathcal{T}_Y \phi X  \quad \text{for} \ X,Y\in \Gamma(\ker F_*).
$$
Substituting $Y$ by $\phi X$ and using (\ref{eq: angle}),
\begin{equation}\label{eq: angle2}
\mathcal{T}_{\phi X} \phi X = -\cos^2 \theta \cdot \mathcal{T}_X X \quad \text{for} \ X\in \Gamma(\ker F_*).
\end{equation}
Similarly, we have

\begin{theorem}\label{thm: angle}
Let $F$ be a semi-slant Riemannian map from an almost Hermitian manifold $(M,g_M,J)$ to a Riemannian manifold $(N,g_N)$ with
the semi-slant angle $\theta$. Then we obtain
\begin{equation}\label{eq: angle3}
\phi^2 X = -\cos^2 \theta\cdot X \quad \text{for} \ X\in \Gamma(\mathcal{D}_2).
\end{equation}
\end{theorem}

\begin{remark}
It is easy to check that the converse of Theorem \ref {thm: angle} is also true.
\end{remark}

Since
\begin{eqnarray*}
g_M(\phi X, \phi Y) & = & \cos^2 \theta g_M(X, Y)  \\
g_M(\omega X, \omega Y) & = & \sin^2 \theta g_M(X, Y)
\end{eqnarray*}
for $ X,Y\in \Gamma(\mathcal{D}_2)$, when $\displaystyle{\theta\in (0, \frac{\pi}{2})}$, we can locally choose an orthonormal frame
$\{ e_1, Je_1, \cdots, e_k, Je_k, f_1, \sec \theta \phi f_1, \csc \theta \omega f_1, \cdots, f_s, \sec \theta \phi f_s,
\csc \theta \omega f_s, g_1, Jg_1, \cdots, g_t,$ $Jg_t \}$
of $TM$ such that $\{ e_1, Je_1, \cdots, e_k, Je_k \}$ is an orthonormal frame
of $\mathcal{D}_1$,  $\{ f_1,$ $\sec \theta \phi f_1, \cdots, f_s, \sec \theta \phi f_s \}$ an orthonormal frame
of $\mathcal{D}_2$, $\{ \csc \theta \omega f_1, \cdots, \csc \theta \omega f_s \}$ an orthonormal frame
of $\omega\mathcal{D}_2$, and $\{ g_1, Jg_1, \cdots, g_t, Jg_t \}$ an orthonormal frame
of $\mu$.

In a similar way, we have

\begin{lemma}\label{lem: angle}
Let $F$ be a semi-slant Riemannian map from a K\"{a}hler manifold $(M,g_M,J)$ to a Riemannian manifold $(N,g_N)$ with
the semi-slant angle $\theta$. If the tensor $\omega$ is parallel, then we get
\begin{equation}\label{eq: angle4}
\mathcal{T}_{\phi X} \phi X = -\cos^2 \theta \cdot \mathcal{T}_X X \quad \text{for} \ X\in \Gamma(\mathcal{D}_2).
\end{equation}
\end{lemma}

We now investigate the integrability of distributions. The proofs of the following Theorems are the same with those of
Theorem 2.3 and Theorem 2.4 in \cite {PP}.

\begin{theorem}
Let $F$ be a semi-slant Riemannian map from an almost Hermitian
manifold $(M,g_M,J)$ to a Riemannian manifold $(N,g_N)$. Then
the complex distribution $\mathcal{D}_1$ is integrable if and only
if we have
$$
\omega (\widehat{\nabla}_X Y-\widehat{\nabla}_Y X) = 0 \quad \text{for} \ X,Y\in
\Gamma(\mathcal{D}_1).
$$
\end{theorem}

Similarly, we get

\begin{theorem}
Let $F$ be a semi-slant Riemannian map from an almost Hermitian
manifold $(M,g_M,J)$ to a Riemannian manifold $(N,g_N)$. Then
the slant distribution $\mathcal{D}_2$ is integrable if and only
if we obtain
$$
P(\phi(\widehat{\nabla}_X Y-\widehat{\nabla}_Y X)) = 0 \quad \text{for} \ X,Y\in
\Gamma(\mathcal{D}_2).
$$
\end{theorem}

Given a semi-slant Riemannian map $F$ from an almost Hermitian
manifold $(M,g_M,J)$ to a Riemannian manifold $(N,g_N)$ with
the semi-slant angle $\displaystyle{\theta\in [0, \frac{\pi}{2})}$, we define an endomorphism
$\widehat{J}$ of $\ker F_*$ by
$$
\widehat{J} := JP+\sec \theta \phi Q.
$$
Then
\begin{equation} \label{eq: compst}
{\widehat{J}}^2 = -id \quad \text{on} \ \ker F_*.
\end{equation}
Note that the distribution $\ker F_*$ is integrable and does not need to be invariant under the almost complex structure $J$.
Furthermore, its dimension may be odd. But with the endomorphism $\widehat{J}$ we have

\begin{theorem}
Let $F$ be a semi-slant Riemannian map from an almost Hermitian
manifold $(M,g_M,J)$ to a Riemannian manifold $(N,g_N)$ with
the semi-slant angle $\displaystyle{\theta\in [0, \frac{\pi}{2})}$.
Then the fibers $(F^{-1}(x), \widehat{J})$ are almost complex manifolds for $x\in M$.
\end{theorem}

We deal with the harmonicity of a map $F$. Given a $C^{\infty}$-map $F$ from a Riemannian manifold $(M,g_M)$ to a Riemannian manifold
$(N,g_N)$, we can naturally define a function $e(F) : M\mapsto [0, \infty)$ given by
$$
e(F)(x) := \frac{1}{2} |(F_*)_x|^2, \quad x\in M,
$$
where $|(F_*)_x|$ denotes the Hilbert-Schmidt norm of $(F_*)_x$ \cite{BW}. We call $e(F)$ the {\em energy density} of $F$.
Let $K$ be a compact domain of $M$, i.e., $K$ is the compact closure $\bar{U}$ of a non-empty connected open subset $U$ of $M$.
The {\em energy integral} of $F$ over $K$ is the integral of its energy density:
$$
E(F;K) := \int_K e(F) v_{g_M} = \frac{1}{2} \int_K |F_*|^2 v_{g_M},
$$
where $v_{g_M}$ is the volume form on $(M,g_M)$.
Let $C^{\infty}(M,N)$ denote the space of all $C^{\infty}$-maps from $M$ to $N$. A $C^{\infty}$-map $F : M\mapsto N$ is said to be
{\em harmonic} if it is a critical point of the energy functional $E(\ ;K) : C^{\infty}(M,N)\mapsto \mathbb{R}$ for any compact domain $K\subset M$.
By the result of J. Eells and J. Sampson \cite{ES}, we know that the map $F$ is harmonic if and only if the tension field $\tau (F) := trace \nabla F_* = 0$.

\begin{theorem}
Let $F$ be a semi-slant Riemannian map from a K\"{a}hler manifold $(M,g_M,J)$ to a Riemannian manifold $(N,g_N)$ such that
$\mathcal{D}_1$ is integrable. Then $F$ is harmonic if and only if $trace (\nabla F_*) = 0$ on $\mathcal{D}_2$ and $\widetilde{H} = 0$,
where $\widetilde{H}$ denotes the mean curvature vector field of $range F_*$.
\end{theorem}

\begin{proof}
Using Lemma \ref{lem: hori}, we have $trace \nabla F_*|_{\ker F_*}\in \Gamma(range F_*)$ and $trace \nabla F_*|_{(\ker F_*)^{\perp}}$
$\in \Gamma((range F_*)^{\perp})$ so that
$$
trace (\nabla F_*) = 0 \quad \Leftrightarrow \quad trace \nabla F_*|_{\ker F_*} = 0 \ \text{and} \ trace \nabla F_*|_{(\ker F_*)^{\perp}} = 0.
$$
Since $\mathcal{D}_1$ is invariant under $J$, we can choose locally an orthonormal frame $\{ e_1, Je_1,$ $\cdots, e_k, Je_k \}$ of $\mathcal{D}_1$.
Using the integrability of the distribution $\mathcal{D}_1$,
\begin{eqnarray*}
(\nabla F_*)(Je_i,Je_i) & = & -F_* \nabla_{Je_i} Je_i = -F_* J(\nabla_{e_i} Je_i+[Je_i, e_i])    \\
                        & = & F_* \nabla_{e_i} e_i = -(\nabla F_*)(e_i,e_i) \quad \text{for} \ 1\leq i\leq k.
\end{eqnarray*}
Hence,
$$
trace \nabla F_*|_{\ker F_*} = 0 \quad \Leftrightarrow \quad trace \nabla F_*|_{\mathcal{D}_2} = 0.
$$
Moreover, it is easy to get that
$$
trace \nabla F_*|_{(\ker F_*)^{\perp}} = l\widetilde{H} \quad \text{for} \ l := \dim (\ker F_*)^{\perp}
$$
so that
$$
trace \nabla F_*|_{(\ker F_*)^{\perp}} = 0 \quad \Leftrightarrow \quad \widetilde{H} = 0.
$$
Therefore, we obtain the result.
\end{proof}

Using Lemma \ref {lem: angle}, we have

\begin{corollary}
Let $F$ be a semi-slant Riemannian map from a K\"{a}hler manifold $(M,g_M,J)$ to a Riemannian manifold $(N,g_N)$ such that
$\mathcal{D}_1$ is integrable and the semi-slant angle $\displaystyle{\theta\in [0, \frac{\pi}{2})}$.
Assume that the tensor $\omega$ is parallel.
Then $F$ is harmonic if and only if $\widetilde{H} = 0$.
\end{corollary}

We now study the condition for such a map $F$ to be totally geodesic.

\begin{theorem}
Let $F$ be a  semi-slant Riemannian map from a K\"{a}hler manifold
$(M,g_M,J)$ to a Riemannian manifold $(N,g_N)$. Then $F$ is a
totally geodesic map if and only if
\begin{align*}
  &\omega(\widehat{\nabla}_X \phi Y+\mathcal{T}_X \omega Y)+C(\mathcal{T}_X \phi Y
    +\mathcal{H}\nabla_X \omega Y) = 0   \\
  &\omega(\widehat{\nabla}_X BZ+\mathcal{T}_X CZ)+C(\mathcal{T}_X BZ
    +\mathcal{H}\nabla_X CZ) = 0   \\
  &\bar{Q}(\nabla_{Z_1}^F F_* Z_2) = 0
\end{align*}
for $X,Y\in \Gamma(\ker F_*)$ and $Z,Z_1,Z_2\in \Gamma((\ker F_*)^{\perp})$.
\end{theorem}

\begin{proof}
If $Z_1,Z_2\in \Gamma((\ker F_*)^{\perp})$, then by Lemma \ref {lem: hori}, we have
$$
(\nabla F_*)(Z_1,Z_2)=0 \quad \Leftrightarrow \quad \bar{Q}((\nabla F_*)(Z_1,Z_2)) = \bar{Q}(\nabla_{Z_1}^F F_* Z_2) = 0.
$$
Given $X,Y\in \Gamma(\ker F_*)$, we get
\begin{align*}
(\nabla F_*)(X,Y)&= -F_* (\nabla_X Y) = F_* (J\nabla_X (\phi Y+\omega Y))   \\
          &= F_* (\phi \widehat{\nabla}_X \phi Y+\omega\widehat{\nabla}_X \phi Y+B\mathcal{T}_X \phi Y
             +C\mathcal{T}_X \phi Y+\phi \mathcal{T}_X \omega Y+\omega \mathcal{T}_X \omega Y   \\
          & \ \ \  +B\mathcal{H}\nabla_X \omega Y+C\mathcal{H}\nabla_X \omega Y).
\end{align*}
Hence,
$$
(\nabla F_*)(X,Y) = 0 \Leftrightarrow \omega(\widehat{\nabla}_X \phi
Y+\mathcal{T}_X \omega Y)+C(\mathcal{T}_X \phi Y
    +\mathcal{H}\nabla_X \omega Y) = 0.
$$
If $X\in \Gamma(\ker F_*)$ and $Z\in \Gamma((\ker F_*)^{\perp})$, then since the tensor $\nabla F_*$ is symmetric,
we only need to consider the following:
\begin{align*}
(\nabla F_*)(X,Z)&= -F_* (\nabla_X Z) = F_* (J\nabla_X (BZ+CZ))   \\
          &= F_* (\phi \widehat{\nabla}_X BZ+\omega\widehat{\nabla}_X BZ+B\mathcal{T}_X BZ
             +C\mathcal{T}_X BZ+\phi \mathcal{T}_X CZ+\omega \mathcal{T}_X CZ   \\
          & \ \ \  +B\mathcal{H}\nabla_X CZ+C\mathcal{H}\nabla_X CZ).
\end{align*}
Thus,
$$
(\nabla F_*)(X,Z) = 0 \Leftrightarrow \omega(\widehat{\nabla}_X
BZ+\mathcal{T}_X CZ)+C(\mathcal{T}_X BZ
    +\mathcal{H}\nabla_X CZ) = 0.
$$
Therefore, we obtain the result.
\end{proof}

Let $F : (M,g_M)\mapsto (N,g_N)$ be a Riemannian map. The map
$F$ is called a Riemannian map {\em with totally umbilical
fibers} if
\begin{equation}\label{eq: umbil}
\mathcal{T}_X Y = g_M (X, Y)H \quad \text{for} \ X,Y\in \Gamma(\ker F_*),
\end{equation}
where $H$ is the mean curvature vector field of the fiber.

In the same way with the proof of Lemma 2.19 in \cite {PP}, we can show

\begin{lemma}
Let $F$ be a semi-slant Riemannian map with totally umbilical fibers
from a K\"{a}hler manifold $(M,g_M,J)$ to a Riemannian manifold
$(N,g_N)$. Then we have
$$
H\in \Gamma(\omega \mathcal{D}_2).
$$
\end{lemma}

\section{Decomposition theorems}\label{decom}

Given a Riemannian manifold $(M,g_M)$, we consider a distribution $\mathcal{D}$ on $M$. We call the distribution $\mathcal{D}$
{\em autoparallel} (or a {\em totally geodesic foliation}) if $\nabla_X Y\in \Gamma(\mathcal{D})$ for $X,Y\in \Gamma(\mathcal{D})$.
If $\mathcal{D}$ is autoparallel, then it is obviously integrable and its leaves are totally geodesic in $M$. The distribution
$\mathcal{D}$ is said to be {\em parallel} if $\nabla_Z Y\in \Gamma(\mathcal{D})$ for $Y\in \Gamma(\mathcal{D})$ and $Z\in \Gamma(TM)$.
If $\mathcal{D}$ is parallel, then we easily obtain that its orthogonal complementary distribution $\mathcal{D}^{\perp}$ is also parallel.
In this situation, $M$ is locally a Riemannian product manifold of the leaves of $\mathcal{D}$ and $\mathcal{D}^{\perp}$.
It is also easy to show that if the distributions $\mathcal{D}$ and $\mathcal{D}^{\perp}$ are simultaneously autoparallel, then
they are also parallel.

\begin{theorem}
Let $F$ be a semi-slant Riemannian map from a K\"{a}hler manifold $(M,g_M,J)$ to a Riemannian manifold $(N,g_N)$.
Then $(M,g_M,J)$ is locally a Riemannian product manifold of the leaves of $\ker F_*$ and $(\ker F_*)^{\perp}$
if and only if
$$
\omega (\widehat{\nabla}_X \phi Y+\mathcal{T}_X \omega
Y)+C(\mathcal{T}_X \phi Y+\mathcal{H}\nabla_X \omega Y) = 0 \quad
\text{for} \ X,Y\in \Gamma(\ker F_*)
$$
and
$$
\phi(\mathcal{V}{\nabla}_Z BW+\mathcal{A}_Z CW)+B(\mathcal{A}_Z
BW+\mathcal{H}\nabla_Z CW) = 0 \quad \text{for} \ Z,W\in
\Gamma((\ker F_*)^{\perp}).
$$
\end{theorem}

\begin{proof}
For $X,Y\in \Gamma(\ker F_*)$,
\begin{align*}
\nabla_X Y&= -J\nabla_X JY= -J(\widehat{\nabla}_X \phi Y+\mathcal{T}_X \phi Y+\mathcal{T}_X \omega Y
             +\mathcal{H}\nabla_X \omega Y)   \\
          &= -(\phi \widehat{\nabla}_X \phi Y+\omega\widehat{\nabla}_X \phi Y+B\mathcal{T}_X \phi Y
             +C\mathcal{T}_X \phi Y+\phi \mathcal{T}_X \omega Y+\omega \mathcal{T}_X \omega Y    \\
          & \ \ \  +B\mathcal{H}\nabla_X \omega Y+C\mathcal{H}\nabla_X \omega
             Y).
\end{align*}
Thus,
$$
\nabla_X Y\in \Gamma(\ker F_*) \Leftrightarrow
\omega(\widehat{\nabla}_X \phi Y+\mathcal{T}_X \omega
Y)+C(\mathcal{T}_X \phi Y+\mathcal{H}\nabla_X \omega Y) =0.
$$
Given $Z,W\in \Gamma((\ker F_*)^{\perp})$, we have
\begin{align*}
\nabla_Z W & = -J\nabla_Z JW = -J(\mathcal{V}\nabla_Z BW+\mathcal{A}_Z
BW+\mathcal{A}_Z CW+\mathcal{H}\nabla_Z CW)   \\
           & = -(\phi \mathcal{V}\nabla_Z BW+\omega\mathcal{V}\nabla_Z BW+B\mathcal{A}_Z BW
             +C\mathcal{A}_Z BW+\phi \mathcal{A}_Z CW   \\
           &  \ \ \ +\omega \mathcal{A}_Z CW  +B\mathcal{H}\nabla_Z CW+C\mathcal{H}\nabla_Z CW).
\end{align*}
Hence,
$$
\nabla_Z W\in \Gamma((\ker F_*)^{\perp}) \Leftrightarrow
\phi(\mathcal{V}{\nabla}_Z BW+\mathcal{A}_Z CW)+B(\mathcal{A}_Z
BW+\mathcal{H}\nabla_Z CW) = 0.
$$
Therefore, the result follows.
\end{proof}

\begin{theorem}
Let $F$ be a semi-slant Riemannian map from a K\"{a}hler manifold $(M,g_M,J)$ to a Riemannian manifold $(N,g_N)$.
Then the fibers of $F$ are locally Riemannian product manifolds of the leaves of $\mathcal{D}_1$ and $\mathcal{D}_2$
if and only if
$$
Q(\phi\widehat{\nabla}_U \phi V+B\mathcal{T}_U \phi V) = 0 \
\text{and} \ \omega\widehat{\nabla}_U \phi V+C\mathcal{T}_U \phi V =
0 \quad \text{for} \ U,V\in \Gamma(\mathcal{D}_1)
$$
and
\begin{align*}
  &P(\phi(\widehat{\nabla}_X \phi Y+\mathcal{T}_X \omega Y)+B(\mathcal{T}_X \phi Y
    +\mathcal{H}\nabla_X \omega Y)) = 0     \\
  &\omega(\widehat{\nabla}_X \phi Y+\mathcal{T}_X \omega Y)+C(\mathcal{T}_X \phi Y
    +\mathcal{H}\nabla_X \omega Y) = 0
\end{align*}
for $X,Y\in \Gamma(\mathcal{D}_2)$.
\end{theorem}

\begin{proof}
Given $U,V\in \Gamma(\mathcal{D}_1)$, we get
\begin{align*}
\nabla_U V&= -J\nabla_U JV= -J(\widehat{\nabla}_U \phi V+\mathcal{T}_U \phi V)   \\
          &= -(\phi \widehat{\nabla}_U \phi V+\omega\widehat{\nabla}_U \phi V+B\mathcal{T}_U \phi V
             +C\mathcal{T}_U \phi V).
\end{align*}
Hence,
$$
\nabla_U V\in \Gamma(\mathcal{D}_1) \Leftrightarrow Q(\phi
\widehat{\nabla}_U \phi V+B\mathcal{T}_U \phi V) = 0 \ \text{and} \
\omega\widehat{\nabla}_U \phi V+C\mathcal{T}_U \phi V = 0.
$$
For $X,Y\in \Gamma(\mathcal{D}_2)$, we obtain
\begin{align*}
\nabla_X Y&= -J\nabla_X JY= -J(\widehat{\nabla}_X
\phi Y+\mathcal{T}_X \phi Y+\mathcal{T}_X \omega Y+\mathcal{H}\nabla_X \omega Y)   \\
          &= -(\phi \widehat{\nabla}_X \phi Y+\omega\widehat{\nabla}_X \phi Y+B\mathcal{T}_X \phi Y
             +C\mathcal{T}_X \phi Y+\phi \mathcal{T}_X \omega Y   \\
          & \ \ \ +\omega \mathcal{T}_X \omega Y  +B\mathcal{H}\nabla_X \omega Y+
          C\mathcal{H}\nabla_X \omega Y).
\end{align*}
Thus,

$\nabla_X Y\in \Gamma(\mathcal{D}_2) \Leftrightarrow$
$$
  \begin{cases}
     P(\phi(\widehat{\nabla}_X \phi Y+\mathcal{T}_X \omega Y)+B(\mathcal{T}_X \phi Y
    +\mathcal{H}\nabla_X \omega Y)) = 0,& \\
     \omega(\widehat{\nabla}_X \phi Y+\mathcal{T}_X \omega Y)+C(\mathcal{T}_X \phi Y
    +\mathcal{H}\nabla_X \omega Y) = 0.&
  \end{cases}
$$
Therefore, we have the result.
\end{proof}

\section{Examples}\label{exam}

Note that given an Euclidean space $\mathbb{R}^{2n}$ with coordinates $(x_1,x_2,\cdots,x_{2n})$, we can canonically choose an
almost complex structure $J$ on $\mathbb{R}^{2n}$ as follows:
\begin{align*}
&J(a_1\frac{\partial}{\partial x_1}+a_2\frac{\partial}{\partial x_2}+\cdots+a_{2n-1}\frac{\partial}{\partial x_{2n-1}}+a_{2n}\frac{\partial}{\partial x_{2n}}) \\
&= -a_2\frac{\partial}{\partial x_1}+a_1\frac{\partial}{\partial x_2}+\cdots-a_{2n}\frac{\partial}{\partial x_{2n-1}}+a_{2n-1}\frac{\partial}{\partial x_{2n}},
\end{align*}
where $a_1, \cdots, a_{2n}\in \mathbb{R}$.
Throughout this section, we will use this notation.

\begin{example}
Let $F$ be an almost Hermitian submersion from an almost Hermitian manifold
$(M,g_M,J_M)$ onto an almost Hermitian manifold $(N,g_N,J_N)$ \cite{W}. Then the
map $F$ is a semi-slant Riemannian map with $\mathcal{D}_1 = \ker F_*$.
\end{example}

\begin{example}
Let $F$ be a slant submersion from an almost Hermitian manifold
$(M,g_M,J)$ onto a Riemannian manifold $(N,g_N)$ with the slant angle $\theta$ \cite{S}. Then the
map $F$ is a semi-slant Riemannian map such that $\mathcal{D}_2 = \ker F_*$ and the semi-slant angle $\theta$.
\end{example}

\begin{example}
Let $F$ be an anti-invariant submersion from an almost Hermitian manifold
$(M,g_M,J)$ onto a Riemannian manifold $(N,g_N)$ \cite{S3}. Then the
map $F$ is a semi-slant Riemannian map such that $\mathcal{D}_2 = \ker F_*$ and the semi-slant angle $\theta = \frac{\pi}{2}$.
\end{example}

\begin{example}
Let $F$ be a semi-invariant submersion from an almost Hermitian
manifold $(M,g_M,J)$ onto a Riemannian manifold $(N,g_N)$ \cite{S2}.
Then the map $F$ is a semi-slant Riemannian map with the semi-slant
angle $\theta=\frac{\pi}{2}$.
\end{example}

\begin{example}
Let $F$ be a semi-slant submersion from an almost Hermitian
manifold $(M,g_M,J)$ onto a Riemannian manifold $(N,g_N)$ with the semi-slant
angle $\theta$ \cite{PP}.
Then the map $F$ is a semi-slant Riemannian map with the semi-slant
angle $\theta$.
\end{example}

\begin{example}
Let $(M,g_M,J)$ be a $2m$-dimensional almost Hermitian manifold and $(N,g_N)$ a $(2m-1)$-dimensional Riemannian manifold.
Let $F$ be a Riemannian map from an almost Hermitian
manifold $(M,g_M,J)$ to a Riemannian manifold $(N,g_N)$ with $\rank F = 2m-1$.
Then the map $F$ is a semi-slant Riemannian map such that
$\mathcal{D}_2 = \ker F_*$ and the semi-slant
angle $\theta=\frac{\pi}{2}$.
\end{example}

\begin{example}
Define a map $F : \mathbb{R}^8 \mapsto \mathbb{R}^5$ by
$$
F(x_1,x_2,\cdots, x_8) = (x_2,x_1,\frac{x_5\cos \alpha+x_6 \sin \alpha+x_4}{\sqrt{2}},0,x_5\sin \alpha-x_6\cos \alpha)
$$
with $\alpha\in (0, \frac{\pi}{2})$.
Then the map $F$ is a semi-slant Riemannian map such that
$$
\mathcal{D}_1 = <\frac{\partial}{\partial x_7},
\frac{\partial}{\partial x_8}> \ \text{and} \ \mathcal{D}_2 =
<\frac{\partial}{\partial x_3}, \cos \alpha \frac{\partial}{\partial
x_5}+\sin \alpha \frac{\partial}{\partial x_6}-\frac{\partial}{\partial x_4}>
$$
with the semi-slant angle $\theta=\frac{\pi}{4}$.
\end{example}

\begin{example}
Define a map $F : \mathbb{R}^6 \mapsto \mathbb{R}^3$ by
$$
F(x_1,x_2,\cdots, x_6) = (x_1 \cos \alpha-x_3 \sin \alpha,c,x_4),
$$
where $\alpha\in (0,\frac{\pi}{2})$ and $c\in\mathbb{R}$. Then the map $F$ is a
semi-slant Riemannian map such that
$$
\mathcal{D}_1 = <\frac{\partial}{\partial x_5},
\frac{\partial}{\partial x_6}> \ \text{and} \ \mathcal{D}_2 =
<\frac{\partial}{\partial x_2}, \sin \alpha \frac{\partial}{\partial
x_1}+\cos \alpha \frac{\partial}{\partial x_3}>
$$
with the semi-slant angle $\theta=\alpha$.
\end{example}

\begin{example}
Define a map $F : \mathbb{R}^{10} \mapsto \mathbb{R}^7$ by
$$
F(x_1,x_2,\cdots, x_{10}) =
(x_4,0,x_3,\frac{x_5-x_6}{\sqrt{2}},0,\frac{x_7+x_9}{\sqrt{2}},\frac{x_8+x_{10}}{\sqrt{2}}).
$$
Then the map $F$ is a semi-slant Riemannian map such that
$$
\mathcal{D}_1 = <\frac{\partial}{\partial x_1},
\frac{\partial}{\partial x_2},-\frac{\partial}{\partial
x_7}+\frac{\partial}{\partial x_9},-\frac{\partial}{\partial
x_8}+\frac{\partial}{\partial x_{10}}> \ \text{and} \ \mathcal{D}_2
= <\frac{\partial}{\partial x_5}+\frac{\partial}{\partial x_6}>
$$
with the semi-slant angle $\theta=\frac{\pi}{2}$.
\end{example}

\begin{example}
Define a map $F : \mathbb{R}^{10} \mapsto \mathbb{R}^5$ by
$$
F(x_1,x_2,\cdots, x_{10}) =
(\frac{x_3+x_5}{\sqrt{2}},2012,x_6,\frac{x_7+x_9}{\sqrt{2}},x_8).
$$
Then the map $F$ is a semi-slant Riemannian map such that
$$
\mathcal{D}_1 = <\frac{\partial}{\partial x_1},
\frac{\partial}{\partial x_2}> \ \text{and} \ \mathcal{D}_2 =
<\frac{\partial}{\partial x_3}-\frac{\partial}{\partial x_5},
\frac{\partial}{\partial x_7}-\frac{\partial}{\partial
x_9},\frac{\partial}{\partial x_4},\frac{\partial}{\partial x_{10}}>
$$
with the semi-slant angle $\theta=\frac{\pi}{4}$.
\end{example}

\begin{example}
Define a map $F : \mathbb{R}^8 \mapsto \mathbb{R}^5$ by
$$
F(x_1,x_2,\cdots, x_8) = (x_8,x_7,\gamma,x_3 \cos \alpha-x_5 \sin
\alpha,x_4\sin \beta-x_6\cos \beta),
$$
where $\alpha$, $\beta$, $\gamma$ are constant. Then the map $F$ is a
semi-slant Riemannian map such that
$$
\mathcal{D}_1 = <\frac{\partial}{\partial
x_1},\frac{\partial}{\partial x_2}> \ \text{and} \ \mathcal{D}_2 =
<\sin \alpha\frac{\partial}{\partial x_3}+\cos
\alpha\frac{\partial}{\partial x_5}, \cos
\beta\frac{\partial}{\partial x_4}+\sin
\beta\frac{\partial}{\partial x_6}>
$$
with the semi-slant angle $\theta$ with $\cos\theta=|\sin
(\alpha+\beta)|$.
\end{example}

\begin{example}
Let $\widetilde{F}$ be a slant Riemannian map from an almost Hermitian manifold
$(M_1,g_{M_1},J_1)$ to a Riemannian manifold $(N,g_N)$ with the
slant angle $\theta$ \cite{S4} and $(M_2,g_{M_2},J_2)$ an almost Hermitian
manifold. Let $(M,g,J)$ be the warped product of
$(M_1,g_{M_1},J_1)$ and $(M_2,g_{M_2},J_2)$ by a positive function
$f$ on $M_1$ \cite {FIP}, where $J=J_1\times J_2$. Define a map $F :
 (M,g,J)\mapsto (N,g_N)$ by
$$
F(x,y) = \widetilde{F}(x) \quad \text{for} \ x\in M_1 \ \text{and} \ y\in M_2.
$$
Then the map $F$ is a semi-slant Riemannian map such that
$\mathcal{D}_1 = TM_2$ and $\mathcal{D}_2 = \ker \widetilde{F}_*$ with the semi-slant angle $\theta$.
\end{example}



\end{document}